\documentclass[12pt]{article}

\usepackage{latexsym,amssymb,upref,amsmath,amsthm, amsfonts,authblk}
\usepackage{amssymb,amsmath,amsthm, calc, graphicx}
\usepackage{epsfig}
\usepackage{breqn}
\usepackage{footnpag}
\usepackage{rotating}
\usepackage{amsfonts}
\usepackage{setspace}
\usepackage{fullpage}
\usepackage{enumitem}
\usepackage{bbold}
\usepackage{comment}
\usepackage{pgf,tikz}
\usepackage{mathrsfs}
\usetikzlibrary{arrows}
\usepackage{yhmath}
\usepackage{hyperref}
\usepackage{authblk}
\usepackage{mathrsfs}

\newtheorem{thm}{Theorem}
\newtheorem{definition}{Definition}
\newtheorem{claim}[thm]{Claim}

\newtheorem{corollary}[thm]{Corollary}

\newcommand{\thistheoremname}{}
\newtheorem*{genericthm*}{\thistheoremname}
\newenvironment{namedthm*}[1]
{\renewcommand{\thistheoremname}{#1}%
	\begin{genericthm*}}
	{\end{genericthm*}}

\newcommand\ex{\ensuremath{\mathrm{ex}}}
\newcommand\lin{\ensuremath{\mathrm{lin}}}
\newcommand\bip{\ensuremath{\mathrm{bip}}}

\newcommand{\C}{\mathscr{C}}

\newcommand{\abs}[1]{\left\lvert{#1}\right\rvert}
\newcommand{\floor}[1]{\left\lfloor{#1}\right\rfloor}

\title{Asymptotics for Tur\'an numbers of cycles in \\ $3$-uniform linear hypergraphs}
\author{
	Beka Ergemlidze\thanks{
		Department of Mathematics, Central European University, Budapest.
		E-mail: beka.ergemlidze@gmail.com
	} \qquad Ervin Gy\H{o}ri \thanks{R\'enyi Institute, Hungarian Academy of Sciences and
	Department of Mathematics, Central European University, Budapest. E-mail: gyori.ervin@renyi.mta.hu} \qquad Abhishek Methuku \thanks{Department of Mathematics, Central European University, Budapest. (Corresponding author) E-mail: abhishekmethuku@gmail.com}
}

\begin{document}

\maketitle

\begin{abstract}
Let $\mathcal{F}$ be a family of $3$-uniform linear hypergraphs. The \emph{linear Tur\'an number} of $\mathcal F$ is the maximum possible number of edges in a $3$-uniform linear hypergraph on $n$ vertices which contains no member of $\mathcal{F}$ as a subhypergraph. 

In this paper we show that the linear Tur\'an number of the five cycle $C_5$ (in the Berge sense) is $\frac{1}{3 \sqrt3}n^{3/2}$ asymptotically. We also show that the linear Tur\'an number of the four cycle $C_4$ and $\{C_3, C_4\}$ are equal asmptotically, which is a strengthening of a theorem of Lazebnik and Verstra\"ete \cite{Lazeb_Verstraete}.


We establish a connection between the linear Tur\'an number of the linear cycle of length $2k+1$ and the extremal number of edges in a graph of girth more than $2k-2$. Combining our result and a theorem of Collier-Cartaino, Graber and Jiang \cite{C_Graber_Jiang}, we obtain that the linear Tur\'an number of the linear cycle of length $2k+1$ is $\Theta(n^{1+\frac{1}{k}})$ for $k = 2, 3, 4, 6$.

\end{abstract}

\section{Introduction}

A hypergraph $H = (V, E)$ is a family $E$ of distinct subsets of a finite set $V$. The members of $E$ are called \emph{hyperedges} and the elements of $V$ are called \emph{vertices}. A hypergraph is called $r$-uniform is each member of $E$ has size $r$. A hypergraph $H = (V, E)$ is called \emph{linear} if every two hyperedges have at most one vertex in common. A hypergraph is $\mathcal F$-free if it does not contain any member of $\mathcal F$ as a subhypergraph. A $2$-uniform hypergraph is simply called a graph. 

\vspace{3mm}

Given a family of graphs $\mathcal F$, the  \emph{Tur\'an number} of $\mathcal F$, denoted $\ex(n, \mathcal F)$, is the maximum number of edges in an $\mathcal F$-free graph on $n$ vertices and the \emph{bipartite Tur\'an number} of $\mathcal F$, denoted $\ex_{\bip}(n, \mathcal F)$ is the maximum number of edges in an $\mathcal F$-free bipartite graph on $n$ vertices. 

\vspace{3mm}

Given a family of $3$-uniform hypergraphs $\mathcal F$, let $\ex_3(n, \mathcal F)$ denote the maximum number of hyperedges of an $\mathcal F$-free $3$-uniform hypergraph on $n$ vertices and similarly, given a family of $3$-uniform linear hypergraphs $\mathcal F$, the \emph{linear Tur\'an number} of $\mathcal F$, denoted $\ex^{\lin}_3(n, \mathcal F)$, is the maximum number of hyperedges in an $\mathcal F$-free $3$-uniform linear hypergraph on $n$ vertices. When $\mathcal F = \{F\}$ then we simply write $\ex^{\lin}_3(n, F)$ instead of $\ex^{\lin}_3(n, \{F\})$. 

\vspace{3mm}

 
 A (Berge) cycle $C_k$ of length $k \ge 2$ is an alternating sequence of distinct vertices and distinct edges of the form $v_1,h_{1},v_2,h_{2},\ldots,v_k,h_{k}$ where $v_i,v_{i+1} \in h_{i}$ for each $i \in \{1,2,\ldots,k-1\}$ and $v_k,v_1 \in h_{k}$. (Note that forbidding a Berge cycle $C_k$ actually forbids a family of hypergraphs, not just one hypergraph, as there may be many ways to choose the hyperedges $h_i$.) This definition of a hypergraph cycle is the classical definition due to Berge. For $k \ge 2$, F\"uredi and \"Ozkahya \cite{Furedi_Ozkahya} showed $\ex^{\lin}_3(n, C_{2k+1}) \le 2kn^{1+1/k} + 9kn$. In fact it is shown in \cite{Gyori_Lemons, Furedi_Ozkahya} that $\ex_3(n, C_{2k+1}) \le O(n^{1+1/k})$. For the even case it is easy to show $\ex^{\lin}_3(n, C_{2k}) \le \ex(n, C_{2k}) = O(n^{1+1/k})$ by selecting a pair from each hyperedge of a $C_{2k}$-free $3$-uniform linear hypergraph. A (Berge) path of length $k$ is an alternating sequence of distinct vertices and distinct edges of the form $v_0, h_0, v_1,h_{1},v_2,h_{2},\ldots, v_{k-1}, h_{k-1},v_k$ where $v_i,v_{i+1} \in h_{i}$ for each $i \in \{0, 1,2,\ldots,k-1\}$. Recently the notion of Berge cycles and Berge paths was generalized to arbitrary Berge graphs in \cite{gp1} and the linear Tur\'an number of (Berge) $K_{2,t}$ was studied in \cite{T2016} and \cite{GMM}. Below we concentrate on the linear Tur\'an numbers of $C_3$, $C_4$ and $C_5$.
 
\vspace{3mm}

Determining $\ex^{\ensuremath{\mathrm{lin}}}_3(n, C_3)$ is basically equivalent to the famous $(6,3)$-problem which is a special case of a general problem of Brown, Erd\H{o}s, and S\'os \cite{Brown_Erd_Sos}. This was settled by Ruzsa and Szemer\'edi in their classical paper \cite{Ruzsa_Szem}, showing that $n^{2-\frac{c}{\sqrt{\log n}}} < \ex^{\lin}_3(n, C_3) = o(n^2)$  for some constant $c > 0$.

\vspace{3mm}

Only a handful of results are known about the asymptotic behaviour of Tur\'an numbers for hypergraphs. In this paper, we focus on determining the asymptotics of $\ex^{\lin}_3(n, C_5)$ by giving a new construction, and a new proof of the upper bound which introduces some important ideas. We also determine the asymptotics of $\ex^{\lin}_3(n, C_4)$ and construct $3$-uniform linear hypergraphs avoiding linear cycles of given odd length(s). In an upcoming paper  \cite{C_5_C_4}, we focus on estimating $\ex_3(n, C_4)$ and $\ex_3(n, C_5)$, improving an estimate of Bollob\' as and Gy\H{o}ri \cite{BGY2008} that shows $\frac{n^{3/2}}{3\sqrt{3}} \le \ex_3(n, C_5) \le \sqrt{2}n^{3/2} + 4.5n$. Surprisingly, even though the $C_5$-free hypergraph that Bollob\'as and Gy\H{o}ri constructed in order to establish their lower bound has the same size as the $C_5$-free hypergraph we constructed in order to obtain the lower bound in our Theorem \ref{main} below, these two constructions are quite different. Their hypergraph is very far from being linear. 

\vspace{3mm}

 The following is our main result.

\begin{thm}
	\label{main}
	$$\ex^{\lin}_3(n, C_5) = \frac{1}{3 \sqrt3}n^{3/2} + O(n).$$
\end{thm}

To show the lower bound in the above theorem we give the following construction. For the sake of convenience we usually drop floors and ceilings of various quantities in the construction below, and in the rest of the paper, as it does not effect the asymptotics. 
\vspace{1mm}

\textbf{Construction of a $C_5$-free linear hypergraph $H$:} For each $ 1 \le t \le \sqrt{n/3}$, let $L_t = \{l^t_1, l^t_2, \ldots, l^t_{\sqrt{n/3}}\}$ and $R_t = \{r^t_1, r^t_2, \ldots, r^t_{\sqrt{n/3}}\}$. Let $B = \{v_{i, j} \mid 1 \le i, j \le \sqrt{n/3}\}$.  The vertex set of $H$ is $V(H) = \bigcup\limits_{i=1}^{\sqrt{n/3}} (L_{i} \cup R_{i})  \cup B$ and the edge set of $H$ is $E(H) = \{v_{i,j} l^t_i r^t_j \mid v_{i,j} \in B \text{ and } 1 \le t \le \sqrt{n/3}\}$. 

\vspace{1mm}

Clearly $\abs{V(H)} = n$ and $\abs{E(H)} = \frac{n^{3/2}}{3\sqrt{3}}$ and $H$ is linear. It is easy to check that $H$ is $C_5$-free but this is proved in a more general setting in Theorem \ref{construction}.

\vspace{2mm}

Lazebnik and Verstra\"ete \cite{Lazeb_Verstraete} showed that
\begin{equation}
\label{LazebVers}
 \ex^{\lin}_3(n, \{C_3,C_4\}) = \frac{n^{3/2}}{6} + O(n). 
\end{equation}
%
This was remarkable especially considering the fact that the asymptotics for the corresponding extremal function  for graphs $\ex(n, \{C_3,C_4\})$ is not known and is a long standing problem of Erd\H{o}s \cite{Erdos}. 
Erd\H{o}s and Simonovits \cite{Erd_Sim} conjectured that $\ex(n, \{C_3,C_4\}) = \ex_{\bip}(n, C_4)$ while Allen, Keevash, Sudakov, and Verstra\"ete \cite{AKSV} conjectured that this is not true.

\vspace{2mm}

In this paper we strengthen the above mentioned result of Lazebnik and Verstra\"ete \cite{Lazeb_Verstraete}, by showing that their upper bound in \eqref{LazebVers} still holds even if the $C_3$-free condition is dropped. This shows $\ex^{\lin}_3(n, C_4) \sim \ex^{\lin}_3(n, \{C_3,C_4\})$, as detailed below. 
\begin{thm}
	\label{c4}
	$$\ex^{\lin}_3(n, C_4) \le \frac{1}{6} n\sqrt{n+9} + \frac{n}{2} = \frac{n^{3/2}}{6} + O(n).$$
\end{thm}

The lower bound $\ex^{\lin}_3(n, C_4) \ge \frac{1}{6}n^{3/2} - \frac{1}{6} \sqrt{n}$ follows from \eqref{LazebVers}. (Note that the construction from \cite{Lazeb_Verstraete}  showing this lower bound is $C_3$-free as well.)  Therefore, $$\ex^{\lin}_3(n, C_4) = \frac{n^{3/2}}{6} + O(n).$$


\vspace{2mm}

Our last result shows strong connection between Tur\'an numbers of even cycles in graphs and linear Tur\'an numbers of linear cycles of odd length in $3$-uniform hypergraphs. This is explained below, after introducing some definitions.

\vspace{2mm}

 A \emph{linear cycle} $C^{\lin}_k$ of length $k \ge 3$ is an alternating sequence $v_1,h_{1},v_2,h_{2},...,v_{k}, h_k$ of distinct vertices and distinct hyperedges such that $h_i \cap h_{i+1} = \{v_{i+1}\}$ for each $i \in \{1,2,\ldots,k-1\}$, $h_1 \cap h_k = \{v_1\}$ and $h_i \cap h_{j}= \emptyset$ if $1 < \abs{j-i} < k-1$. (A \emph{linear path} can be defined similarly.) The vertices $v_1, v_2, \ldots, v_k$ are called the \emph{basic vertices} of $C^{\lin}_k$ and the graph with the edge set $\{v_1v_2, v_2v_3, \ldots, v_{k-1}v_k, v_kv_1 \}$ is called the \emph{basic cycle} of $C^{\lin}_k$. 

\vspace{2mm}

Let $\C_k$ and $\C^{\lin}_k$ denote the set of (Berge) cycles $C_l$ and the set of linear cycles $C^{\lin}_l$, respectively, where $l$ has the \emph{same parity} as $k$ and $2 \le l \le k$. In particular, in Theorem \ref{construction} we will be interested in the sets $\C_{2k-2} = \{C_2, C_4, C_6, \ldots, C_{2k-2}\}$ and $\C^{\lin}_{2k+1} = \{C^{\lin}_{3}, C^{\lin}_{5}, \ldots, C^{\lin}_{2k+1}\}$. Note that the (Berge) cycle $C_2$ corresponds to two hyperedges that share at least 2 vertices, so a hypergraph is linear if and only if it is $C_2$-free. In particular, for graphs (i.e., $2$-uniform hypergraphs) the $C_2$-free condition 
does not impose any restriction, and there is no difference between a (Berge) cycle $C_l$ and a linear cycle $C^{\lin}_l$.

\vspace{2mm}

Bondy and Simonovits \cite{Bondy_Simonovits} showed that for $k \ge 2$, $\ex(n, C_{2k}) \le c_kn^{1+\frac{1}{k}}$ for all sufficiently large $n$. Improvements to the constant factor $c_k$ are made in \cite{verstraete,pihurko,bukhJiang}. The \emph{girth} of a graph is the length of a shortest cycle contained in the graph. For $k = 2, 3, 5$, constructions of $C_{2k}$-free graphs on $n$ vertices with $\Omega(n^{1+\frac{1}{k}})$ edges are known: Benson \cite{Benson} and Singleton \cite{Singleton} constructed a bipartite $\C_6$-free graph with $(1+o(1))(n/2)^{4/3}$  edges and Benson \cite{Benson} constructed a bipartite $\C_{10}$-free graph with $(1+o(1))(n/2)^{6/5}$ edges. For $k \not \in \{2,3,5\}$ it is not known if the order of magnitude of $\ex(n, C_{2k})$ is $\Theta(n^{1+\frac{1}{k}})$. The best known lower bound is due to Lazebnik, Ustimenko and Woldar \cite{Lazebnik_U_W}, who showed that there exist graphs of girth more than $2k+1$ containing $\Omega(n^{1+\frac{2}{3k-3+\epsilon}})$ edges where $k \ge 2$ is fixed, $\epsilon = 0$ if $k$ is odd and $\epsilon = 1$ if $k$ is even. 

\vspace{2mm}

Recently Collier-Cartaino, Graber and Jiang \cite{C_Graber_Jiang} showed that for all $l \ge 3$, $\ex^{\lin}_3(n, C^{\lin}_l) \le O(n^{1+\frac{1}{\floor{l/2}}})$. In fact, they proved the same upper bound for all $r$-uniform hypergraphs with $r \ge 3$. However, it is not known if $C^{\lin}_l$-free linear $3$-uniform hypergraphs on $n$ vertices with $\Omega(n^{1+\frac{1}{\floor{l/2}}})$ hyperedges exist. It is mentioned in \cite{C_Graber_Jiang} that the best known lower bound 
\begin{equation}
\label{vers}
\ex^{\lin}_3(n, C^{\lin}_l) \ge \Omega(n^{1+\frac{1}{l-1}}),
\end{equation}
was observed by Verstra\"ete, by taking a random subgraph of a Steiner triple system. 

\vspace{2mm}

If $l = 2k+1$ is odd, then we are able to construct a $\C^{\lin}_{2k+1}$-free $3$-uniform linear hypergraph on $n$ vertices with $\Omega(n^{1+\frac{1}{k}})$ hyperedges whenever a $\C_{2k-2}$-free graph with $\Omega(n^{1+\frac{1}{k-1}})$ edges exists. More precisely, we show:

%


\begin{thm}
	\label{construction}
Let $\ex_{\bip}(n, \C_{2k-2}) \ge (1 + o(1)) c \left(\frac{n}{2}\right)^{\alpha}  = \Omega(n^{\alpha})$ for some $c, \alpha > 0$. Then, 

$$\ex^{\lin}_3(n, \C^{\lin}_{2k+1}) \ge (1 + o(1)) \frac{\alpha c}{4 \alpha - 2} \cdot \left(\frac{\alpha -1}{c(2 \alpha -1)}\right)^{1 - \frac{1}{\alpha}} n^{2 - \frac{1}{\alpha}} = \Omega(n^{2 - \frac{1}{\alpha}}).$$
\end{thm}

If $2k-2 =2$, then by definition $\C_{2k-2} = \{C_2\}$, so in this case the $\C_{2k-2}$-free condition does not impose any restriction. Thus in order to bound $\ex_{\bip}(n, \C_{2})$ from below, one can take a complete balanced bipartite graph. Therefore, using $c = 1$ and $ \alpha = 2$ in the above theorem, we get $\ex^{\lin}_3(n, \C^{\lin}_5) \ge (1 +o(1)) \frac{n^{3/2}}{3\sqrt{3}}$. 
Since a $3$-uniform linear hypergraph which is both $C^{\lin}_3$-free and $C^{\lin}_5$-free is (Berge) $C_5$-free, this also provides the desired lower bound in Theorem \ref{main}. As we mentioned before, in the cases $2k-2 = 4, 6, 10$, it is known that $c = 1$ and $\alpha = 1+\frac{1}{k-1}$ by the work of Benson and Singleton and for all $k \ge 2$, it is known that $\alpha = 1+\frac{2}{3k-6+\epsilon}$ by the work of Lazebnik, Ustimenko and Woldar, where $\epsilon = 0$ if $k$ is odd and $\epsilon = 1$ if $k$ is even; so substituting these in Theorem \ref{construction} and combining it with the upper bound of Collier-Cartaino, Graber and Jiang, we get the following corollary.

\begin{corollary}
	\label{smallk}
For $k = 2, 3, 4, 6$, we have $\ex^{\lin}_3(n, \C^{\lin}_{2k+1})  \ge (1 + o(1)) \frac{k}{2} (\frac{n}{k+1})^{1+ \frac{1}{k}} .$

Therefore, in these cases, $$\ex^{\lin}_3(n, \C^{\lin}_{2k+1}) = \Theta(n^{1+ \frac{1}{k}}).$$

Moreover, for $k \ge 2$, we have $$\ex^{\lin}_3(n, \C^{\lin}_{2k+1})  \ge \Omega(n^{1 + \frac{2}{3k-4+\epsilon}}),$$ where $\epsilon = 0$ if $k$ is odd and $\epsilon = 1$ if $k$ is even.
\end{corollary}

The above corollary provides an improvement of the lower bound in \eqref{vers} for linear cycles of odd length.

%
\vspace{3mm}
\textbf{Structure of the paper:}  In the next section we introduce some notation that is used through out the paper. In Section \ref{maintheoremsection}, we prove the upper bound of Theorem \ref{main} and in Section \ref{c4section}, we prove Theorem \ref{c4}. Finally, in Section \ref{theorem4section} we prove Theorem \ref{construction}.

\subsection*{Notation}
	\label {notation}
We introduce some important notation used throughout the paper. Length of a path is the number of edges in the path. 

For convenience, throughout the paper, an edge $\{a,b\}$ of a graph or a pair of vertices $a,b$ is referred to as $ab$. A hyperedge $\{a,b,c\}$ is written simply as $abc$. 

For a hypergraph $H$, let $\partial H = \{ab \mid ab \subset e \in E(H)\}$ denote its \emph{2-shadow} graph. (Notice that the basic cycle of $C^{lin}_k$ is a cycle in the graph $\partial{C^{lin}_k}$.) If $H$ is linear, then $\abs{E(\partial H)} = 3 \abs{E(H)}$. For a hypergraph $H$ and $v \in V(H)$, we denote the degree of $v$ in $H$ by $d(v)$. We write $d^H(v)$ instead of $d(v)$ when it is important to emphasize the underlying hypergraph. 

The \textit{first neighborhood} and \textit{second neighborhood} of $v$ in $H$ are defined as 
$$N^H_1(v) = \{x \in V(H) \setminus \{v \}  \mid v,x \in h \text{ for some } h \in E(H)\}$$
 and 
 $$N^H_2(v) = \{x \in V(H) \setminus (N^H_1(v) \cup \{v\}) \mid \exists h \in E(H) \text{ such that } x \in h \text{ and } h \cap N^H_1(v) \not = \emptyset\}$$
  respectively.

\section{$C_5$-free linear hypergraphs: Proof of the upper bound in Theorem \ref{main}}
\label{maintheoremsection}

Let $H$ be a $3$-uniform linear hypergraph on $n$ vertices containing no $C_5$. Let $d$ and $d_{max}$ denote the average degree and maximum degree of a vertex in $H$, respectively. We will show that we may assume $H$ has minimum degree at least $d/3$. Indeed, if there is a vertex whose degree less than one-third of the average degree in the hypergraph, we delete it and all the hyperedges incident to it. Notice that this will not decrease the average degree. We repeat this procedure as long as we can and eventually we obtain a (non-empty) hypergraph $H'$ with $n' \le n$ vertices and average degree $d' \ge d$ and minimum degree at least $d/3$. It is easy to see that if $d' \le  \sqrt{n'/3 + C}$ then $d \le  \sqrt{n/3 + C}$ (for a constant $C > 0$) proving Theorem \ref{main}. So from now on we will assume $H$ has minimum degree at least $d/3$. Our goal is to upper bound $d$.

The following claim shows that for any vertex $v$, the number of hyperedges $h \in E(H)$ with $\abs{h \cap N^H_1(v)} \ge 2$ is small provided $d(v)$ is small. This is useful for proving Claim \ref{vertices_more_than_edges}. Using this and the fact that the minimum degree is at least $d/3$, we will show in Claim \ref{dmax} that we may assume the maximum degree in $H$ is small.

\begin{claim}
	\label{N1}
	Let $v \in V(H)$. Then the number of hyperedges $h \in E(H)$ with $\abs{h \cap N^H_1(v)} \ge 2$ is at most $6 d(v)$.
\end{claim}

\begin{proof}[Proof of Claim \ref{N1}]
	We construct an auxiliary graph $G_1$ whose vertex set is $N^H_1(v)$ in the following way: From each hyperedge $h \in E(H)$ with $\abs{h \cap N^H_1(v)} \ge 2$ and $v \not \in h$, we select exactly one pair $xy \subset h \cap N^H_1(v)$ arbitrarily. We claim that there is no $7$-vertex path in $G_1$. Suppose for the sake of a contradiction that there is a path $v_1v_2v_3v_4v_5v_6v_7$ in $G_1$. Then, one of the two hyperedges $v_1v_4v$, $v_4v_7v$ is not in $E(H)$ as the hypergraph is linear. Suppose without loss of generality that $v_1v_4v \not \in E(H)$, so there are two different hyperedges $h, h'$ such that $v_1, v \in h$ and $v_4, v \in h'$. These two hyperedges together with the $3$ hyperedges containing $v_1v_2$, $v_2v_3$, $v_3v_4$ create a five cycle in $H$ (note that they are different by our construction), a contradiction. So there is no path on seven vertices in $G_1$ and so by Erd\H{o}s-Gallai theorem, $G_1$ contains at most $\frac{7-2}{2}\abs{V(G_1)} \le 2.5 (2d(v))= 5d(v)$ edges, which implies that the number of hyperedges $h \in E(H)$ with $\abs{h \cap N^H_1(v)} \ge 2$ is at most $5d(v) + d(v) = 6 d(v)$.
\end{proof}

Using the previous claim we will show the following claim. 

\begin{claim}
	\label{vertices_more_than_edges}
Let $v \in V(H)$. Then, $$ \abs{N^H_2(v)} \ge \sum_{x \in N^H_1(v)} d(x) - 18 d(v).$$
\end{claim}

\begin{proof} [Proof of Claim \ref{vertices_more_than_edges}]

%

First we count the number of hyperedges $h \in E(H)$ such that $\abs{h \cap N^H_1(v)} = 1$ and $\abs{h \cap N^H_2(v)} = 2$. Let $G_2 = (N^H_2(v), E(G_2))$ be an auxiliary graph whose edge set $E(G_2) = \{xy \mid \exists h \in E(H), \abs{h \cap N^H_1(v)} = 1,  \abs{h \cap N^H_2(v)} = 2 \text{ and } x,y \in h \cap N^H_2(v)\}$. Let $h_1, h_2, \ldots, h_{d(v)}$ be the hyperedges containing $v$. Now we color an edge $xy \in E(G_2)$ with the color $i$ if $x, y \in h$ and $h \cap h_i \not = \emptyset$. Since the hypergraph is linear this gives a coloring of all the edges of $G_2$. 

\begin{claim}
	\label{norainbow}
	If there are three edges $ab, bc, cd \in E(G_2)$ (where $a$ might be the same as $d$), then the color of $ab$ is the same as the color of $cd$.
\end{claim}
\begin{proof}[Proof of Claim \ref{norainbow}]
	Suppose that they have different colors $i$ and $j$ respectively. Then, the hyperedges in $H$ containing $ab$, $bc$, $cd$, together with $h_i$ and $h_j$ form a five cycle, a contradiction.  
\end{proof}

We claim that $G_2$ is triangle-free. Suppose for the sake of a contradiction that there is a triangle, say $abc$, in $G_2$. Then by Claim \ref{norainbow} it is easy to see that all the edges of this triangle must have the same color, say color $i$. Therefore, at least two of the three hyperedges of $H$ containing $ab$, $bc$, $ca$ must contain the same vertex of $h_i$. This is impossible since $H$ is linear.

We claim that if $v_1v_2v_3\ldots v_k$ is a cycle of length $k \ge 4$ in $G_2$, then every vertex in it has degree exactly $2$. Suppose without loss of generality that $v_3w \in E(G_2)$ where $w \not =v_2$, $w \not = v_4$. Since $G_2$ is triangle free, $w \not = v_1$ and $w \not = v_5$ (note that if $k = 4$, then $v_5=v_1$). By Claim \ref{norainbow}, the color of $v_1v_2$ is the same as the colors of $v_3v_4$ and $v_3w$. Also, the color of $v_4v_5$ is the same as the colors of $v_3w$ and $v_2v_3$. This implies that the edges $v_2v_3$, $v_3w$, $v_3v_4$ must have the same color, which is a contradiction since the hypergraph is linear. 
Thus, $G_2$ is a disjoint union of cycles and trees. So $\abs{E(G_2)} \le \abs{V(G_2)} = \abs{N^H_2(v)}$. 

Since $\sum_{x \in N^H_1(v)} d(x) $ is at most the number of edges in $G_2$ plus three times the number of hyperedges $h \in E(H)$ with $\abs{h \cap N^H_1(v)} \ge 2$, applying Claim \ref{N1} we have $$\sum_{x \in N^H_1(v)} d(x) \le \abs{N^H_2(v)} + 3 (6 d(v)),$$ completing the proof of the claim.
\end{proof}

Using the above claim we will show Theorem \ref{main} holds if $d_{max} > 6d$. We do not optimize the constant multiplying $d$ here.

\begin{claim}
	\label{dmax}
	
We may assume $d_{max} \le 6d$ for large enough $n$ (i.e., whenever $n \ge 34992$).
\end{claim}

\begin{proof}
Suppose that $v \in V(H)$ and $d(v) = d_{max} > 6d$. Recall that $H$ has minimum degree at least $\frac{d}{3}$. Then by Claim \ref{vertices_more_than_edges},  $$\abs{N^H_2(v)} \ge \sum_{x \in N^H_1(v)} d(x) - 18 d(v) \ge \frac{d}{3} \abs{N^H_1(v)} -18d(v)=$$ $$= \frac{d}{3} (2d(v)) -18 d(v) = \left(\frac{2d}{3} -18\right)\cdot d(v) > \left(\frac{2d}{3} -18\right)\cdot  6d \ge 3d^2$$ if $d > 108$. That is, if $d > 108$, then $3d^2 \le \abs{N^H_2(v)} \le n$ which implies that $$\abs{E(H)} = \frac{nd}{3} \le \frac{1}{3 \sqrt3}n^{3/2},$$ as required. On the other hand, if $d \le 108$, then $$\abs{E(H)} = \frac{nd}{3} \le 36n \le \frac{1}{3 \sqrt3} n^{3/2}$$ for $n \ge 34992$, proving Theorem \ref{main}.
\end{proof}

In the next definition, for each hyperedge of $H$ we identify a subhypergraph of $H$ corresponding to this hyperedge. (We will later see that this subhypergraph has a negligible fraction of the hyperedges of $H$.)

\begin{definition}
	\label{garbage}
For $abc \in E(H)$, the subhypergraph $H'_{abc}$ of $H$ consists of the hyperedges $h = uvw \in E(H)$ such that $h \cap \{a,b,c\} = \emptyset$ and  $h$ satisfies at least one of the following properties.

\begin{enumerate}
\item $\exists x \in \{a,b,c\}$ such that $\abs{h \cap N^H_1(x)} \ge 2$.
		
\item $h \cap (N^H_1(a) \cap N^H_1(b) \cap N^H_1(c)) \not  = \emptyset$.
		
\item $\{x,y,z\} =  \{a,b,c\}$ and  $u \in N^H_1(x) \cap N^H_1(y)$ and $v \in N^H_1(z)$. 
\end{enumerate}
\end{definition}

\begin{definition}
	
Let $H_{abc}$ be the subhypergraph of $H$ defined by $V(H_{abc}) = V(H)$ and $E(H_{abc}) = E(H) \setminus E(H'_{abc})$. That is, $H_{abc}$ is the hypergraph obtained after deleting all the hyperedges of $H$ which are in $E(H'_{abc})$.
\end{definition}

The following claim shows that the number of hyperedges in $H'_{abc}$ is small. 


\begin{claim}
	\label{few_garbage_edges}
Let $abc \in E(H)$.  Then $$\abs{E(H'_{abc})}  \le 25d_{max}.$$
\end{claim}

\begin{proof}
By Claim \ref{N1}, the number of hyperedges $h \in E(H)$ satisfying property 1 of Definition \ref{garbage} is at most $$6d(a) + 6d(b) + 6d(c) \le 18d_{max}.$$

Now we estimate the number of hyperedges satisfying property 2 of Definition \ref{garbage} . First let us show that $\abs{N^H_1(a) \cap N^H_1(b) \cap N^H_1(c)} \le 1$ which implies that the number of hyperedges satisfying property 2 of Definition \ref{garbage} is at most $d_{max}$. Assume for the sake of a contradiction that $\{u, v\} \subseteq N^H_1(a) \cap N^H_1(b) \cap N^H_1(c)$. Then by linearity of $H$, it is impossible that $uva, uvb, uvc \in E(H)$. Without loss of generality, assume that $uva \not \in E(H)$. Then it is easy to see that the pairs $ua, av, vc, cb, bu$ are contained in distinct hyperedges by linearity of $H$, creating a $C_5$ in $H$, a contradiction. 

Now we estimate the number of hyperedges satisfying property 3 of Definition \ref{garbage}. Fix $x, y , z$ such that $\{x, y ,z\} = \{a, b, c\}$. We will show that for each $v \in N^H_1(z)$, there is at most one hyperedge containing $v$ and a vertex from $N^H_1(x) \cap N^H_1(y)$.
Assume for the sake of a contradiction that there are two different hyperedges $u_1vw_1, u_2vw_2 \in E(H)$ such that $u_1, u_2 \in N^H_1(x) \cap N^H_1(y)$ and $v \in N^H_1(z)$. Now it is easy to see that the pairs $u_1x, xy, yu_2, u_2v, vu_1$ are contained in five distinct hyperedges since $H$ is linear and  $u_1vw_1, u_2vw_2$ are disjoint from $abc$, so there is a $C_5$ in $H$, a contradiction. So for each choice of $z \in \{a,b,c\}$ the number of hyperedges satisfying property 3 of Definition \ref{garbage} is at most $\abs{N^H_1(z)}$. So the total number of hyperedges satisfying property 3 of Definition \ref{garbage} is at most $$\abs{N^H_1(a)} + \abs{N^H_1(b)} +\abs{N^H_1(c)} \le 2(d(a) + d(b) +d(c)) \le 6d_{max}.$$ 

Adding up these estimates, we get the desired bound in our claim.
\end{proof}

A \emph{$3$-link} in $H$ is a set of $3$ hyperedges $h_1, h_2, h_3 \in E(H)$ such that $h_1 \cap h_2  \not = \emptyset$, $h_2 \cap h_3 \not =\emptyset$ and $h_1 \cap h_3 = \emptyset$. The hyperedges $h_1$ and $h_3$ are called \emph{terminal} hyperedges of this $3$-link. (Notice that a given $3$-link defines four different Berge paths because each end vertex can be chosen in two ways. Also note that a $3$-link is simply the set of hyperedges of a linear path of length three.)  

Given a hypergraph $H$ and $abc \in E(H)$, let $p_{abc}(H)$ denote the number of $3$-links in $H$ in which $abc$ is a terminal hyperedge and let $p(H)$ denote the total number of $3$-links in $H$. Notice $$p(H) = \frac{1}{2} \sum_{abc \in E(H)} p_{abc}(H).$$

In Section \ref{upperonp}, we prove an upper bound on $p(H)$ and in Section \ref{loweronp}, we prove a lower bound on $p(H)$ and combine it with the upper bound to obtain the desired bound on $d$.

\subsection{Upper bounding $p(H)$}
\label{upperonp}
For any given $abc \in E(H)$, the following claim upper bounds the number of $3$-links in $H$ in which $abc$ is a terminal hyperedge by a little bit more than $2\abs{V(H)}$.

\begin{claim}
	\label{upperboundonpaths}
Let $abc \in E(H)$. Then, $$p_{abc}(H) \le 2\abs{V(H)} + 273d_{max}.$$
\end{claim}
\begin{proof}[Proof of Claim \ref{upperboundonpaths}]
	
First we show that most of the $3$-links of $H$ are in $H_{abc}$.
\begin{claim}
	\label{Habcdominates}We have,
$$p_{abc}(H) \le p_{abc}(H_{abc}) + 225 d_{max}.$$
\end{claim}
\begin{proof}
 Consider $h \in E(H) \setminus E(H_{abc}) = E(H'_{abc})$. Note that $h \cap \{a,b,c\} = \emptyset$. The number of $3$-links containing both $abc$ and $h$ is at most $9$ since the number of hyperedges in $H$ that intersect both $h$ and $abc$ is at most $9$ as $H$ is linear. Therefore the total number of $3$-links in $H$ containing $abc$ and a hyperedge of $E(H) \setminus E(H_{abc})$ is at most $9\abs{E(H'_{abc})} \le 9(25 d_{max}) = 225 d_{max}$ by Claim \ref{few_garbage_edges} which implies that $p_{abc}(H) \le p_{abc}(H_{abc}) + 225 d_{max}$, as required.
\end{proof}

For $x \in \{a,b,c\}$, let $H_x$ be a subhypergraph of $H_{abc}$ whose edge set is $E(H_x) = E^x_1 \bigcup E^x_2$ where $E^x_1 = \{h \in E(H_{abc}) \mid x \in h \text{ and } h \not = abc\}$ and $E^x_2 = \{h \in E(H_{abc}) \mid \exists h' \in E^x_1,  x \not \in h \text{ and } h \cap h' \not = \emptyset  \}$ and its vertex set is $V(H_x) = \{v \in V(H_{abc})   \mid \exists h \in E(H_x) \text{ and } v \in h \}$. Note that $\abs{E^x_1} = d^{H_x}(x) = d^H(x)-1$ and every hyperedge in $E^x_1$ contains exactly two vertices of $N^{H_x}_1(x)$ and every hyperedge in $E^x_2$ contains one vertex of $N^{H_x}_1(x)$ and two vertices of $N^{H_x}_2(x)$ because hyperedges containing more than one vertex of $N^{H_x}_1(x)$ do not belong to $H_{abc}$ (since they are in $H'_{abc}$ by property 1 of Definition \ref{garbage}) and thus, do not belong to $H_x$.

We will show that the number of ordered pairs $(x, h)$ such that $x \in \{a,b,c\}$ and $h \in E^x_2$ is equal to $p_{abc}(H_{abc})$ by showing a bijection between the set of ordered pairs $(x, h)$ such that $x \in \{a,b,c\}$ and $h \in E^x_2$ and the set of $3$-links in $H_{abc}$ where $abc$ is a terminal hyperedge. To each $3$-link $abc, h', h$ in $H_{abc}$  where $abc \cap h = \emptyset$ and $h' \cap abc = \{x\}$, let us associate the ordered pair $(x, h)$. Clearly $x \in \{a,b,c\}$ and $h \in E^x_2$. Now consider an ordered pair $(x, h)$ where $x \in \{a,b,c\}$ and $h \in E^x_2$. Then $h$ contains exactly one vertex $u \in N^{H_x}_1(x)$, so there is a unique hyperedge $h' \in E(H)$ containing the pair $ux$. Therefore, there is a unique $3$-link  in $H_{abc}$ associated to $(x, h)$, namely $abc, h', h$, establishing the required bijection. So, 
\begin{equation}
\label{p_abc}
p_{abc}(H_{abc}) = \abs{\{(x, h) \mid x \in \{a,b,c\}, h \in E^x_2 \}} = \sum_{x \in \{a,b,c\}} \abs{E^x_2}.
\end{equation}

Now our aim is to upper bound $p_{abc}(H_{abc})$ in terms of $\sum_{x \in \{a,b,c\}} \abs{N^{H_x}_2(x)}$, which will be upper bounded in Claim \ref{atmost2}.

Substituting $v = x$ and $H = H_x$ in Claim \ref{vertices_more_than_edges}, we get, $\abs{N^{H_x}_2(x)} \ge \sum_{y \in N^{H_x}_1(x)} d^{H_x}(y) - 18 d^{H_x}(x)$ for each $x \in \{a,b,c\}$.  Now since $\sum_{y \in N^{H_x}_1(x)} d(y) = 2 \abs{E^x_1} + \abs{E^x_2}$, we have $\abs{N^{H_x}_2(x)} \ge 2 \abs{E^x_1} + \abs{E^x_2} - 18 d^{H_x}(x)$. So by \eqref{p_abc},
\begin{equation*}
\sum_{x \in \{a,b,c\}} \abs{N^{H_x}_2(x)} \ge \sum_{x \in \{a,b,c\}} (2 \abs{E^x_1} + \abs{E^x_2} - 18 d^{H_x}(x)) = \sum_{x \in \{a,b,c\}} (2 \abs{E^x_1} - 18 d^{H_x}(x)) + p_{abc}(H_{abc}).
\end{equation*}

Since $\abs{E^x_1} = d^{H_x}(x) = d^H(x)-1$, we have $2 \abs{E^x_1} - 18 d^{H_x}(x) = -16 (d^H(x)-1) $. So,

\begin{equation}
\label{second_level}
\sum_{x \in \{a,b,c\}} \abs{N^{H_x}_2(x)} \ge -16 \sum_{x \in \{a,b,c\}} (d^H(x)-1) + p_{abc}(H_{abc}) \ge -48(d_{max}-1) + p_{abc}(H_{abc}).
\end{equation}

Now we want to upper bound $\sum_{x \in \{a,b,c\}} \abs{N^{H_x}_2(x)}$ by $2 \abs{V(H)}$. 

\begin{claim}
	\label{atmost2}
Each vertex $v \in V(H)$ belongs to at most two of the sets $N^{H_a}_2(a), N^{H_b}_2(b), N^{H_c}_2(c)$. So $$\sum_{x \in \{a,b,c\}} \abs{N^{H_x}_2(x)} \le 2 \abs{V(H)}.$$
\end{claim} 
\begin{proof}
Suppose for the sake of a contradiction that there exists a vertex $v \in V(H)$ which is in all three sets $N^{H_a}_2(a), N^{H_b}_2(b), N^{H_c}_2(c)$. Then for each $x \in \{a,b,c\}$, there exists $h_x \in E^x_2$ such that $v \in h_x$. 

First let us assume $h_a = h_b = h_c = h$ and let $h_x \cap N^{H_x}_1(x) = \{v_x\}$ for each $x \in \{a,b,c\}$. If $v_a = v_b = v_c$ then $h \cap (N^H_1(a) \cap N^H_1(b) \cap N^H_1(c)) \not  = \emptyset$, so by property 2 of Definition \ref{garbage}, $h \in E(H'_{abc})$ so $h \not \in E(H_{abc}) \supseteq E^x_2$, a contradiction.  If $v_x = v_y \not = v_z$ for some $\{x,y,z\} = \{a,b,c\}$ then by property 3 of Definition \ref{garbage}, $h \not \in E(H_{abc}) \supseteq E^x_2$, a contradiction again. Therefore, $v_a, v_b, v_c$ are distinct.
Moreover, for each $x \in \{a,b,c\}$, $v_x \in N^{H_x}_1(x)$ and $v \in N^{H_x}_2(x)$. However, since $N^{H_x}_1(x)$ and $N^{H_x}_2(x)$ are disjoint for each $x \in \{a,b,c\}$ by definition (see the Notation section for the precise definition of first and second neighborhoods), $v$ is different from $v_a, v_b$ and $v_c$. So $v, v_a, v_b, v_c \in h$, a contradiction since $h$ is a hyperedge of size $3$. 

So there exist $x, y \in \{a,b,c\}$ such that $h_x \not = h_y$. Also, there exist $h'_x \in E^x_1, h'_y \in E^y_1$ such that $h_x \cap h'_x \not = \emptyset$ and $h_y \cap h'_y \not = \emptyset$. Now it is easy to see that the hyperedges $h_x, h_y, h'_x, h'_y, abc$ form a $C_5$, a contradiction, proving the claim.
\end{proof}

So by Claim \ref{atmost2}, $\sum_{x \in \{a,b,c\}} \abs{N^{H_x}_2(x)}\le 2 \abs{V(H)}$. Combining this with \eqref{second_level}, we get 

\begin{equation}
p_{abc}(H_{abc}) -48(d_{max}-1) \le \sum_{x \in \{a,b,c\}} \abs{N^{H_x}_2(x)} \le  2 \abs{V(H)}.
\end{equation}

Therefore, by Claim \ref{Habcdominates} and the above inequality, we have 
$$p_{abc}(H) \le p_{abc}(H_{abc}) + 225 d_{max} \le 2 \abs{V(H)} + 48(d_{max}-1) + 225 d_{max} \le 2\abs{V(H)} + 273d_{max},$$ completing the proof of Claim \ref{upperboundonpaths}.
\end{proof}

\vspace{3mm}

So by Claim \ref{upperboundonpaths}, we have 

\begin{equation}
p(H) = \frac{1}{2} \sum_{abc \in E(H)} p_{abc}(H) \le \frac{1}{2}  (2\abs{V(H)} + 273d_{max})\abs{E(H)}.
\end{equation}

By Claim \ref{dmax}, we can assume $d_{max} \le 6d$. Using this in the above inequality we obtain, 

\begin{equation}
\label{upper_bound}
p(H)  \le \frac{1}{2}  (2\abs{V(H)} + 1638d)\abs{E(H)} = (n+819d)\frac{nd}{3}.
\end{equation}

\subsection{Lower bounding $p(H)$}
\label{loweronp}

We introduce some definitions that are needed in the rest of our proof where we establish a lower bound on $p(H)$ and combine it with the upper bound in \eqref{upper_bound}.

\vspace{2mm}

A \emph{walk} of length $k$ in a graph is a sequence $v_0e_0v_1e_1 \ldots v_{k-1}e_{k-1}v_k$ of vertices and edges such that $e_i = v_iv_{i+1}$ for $0 \le i < k$. For convenience we simply denote such a walk by $v_0v_1 \ldots v_{k-1} v_k$. A walk is called unordered if $v_0v_1 \ldots v_{k-1} v_k$ and $v_kv_{k-1} \ldots v_1 v_0$ are considered as the same walk. From now on, unless otherwise stated, we only consider unordered walks. A \emph{path} is a walk with no repeated vertices or edges. Blakley and Roy \cite{Blakley_Roy} proved a matrix version of H\"older's inequality, which implies that any graph $G$ with average degree $d^G$ has at least as many walks of a given length as a $d^G$-regular graph on the same number of vertices.

\vspace{2mm}

We will now prove a lower bound on $p(H)$. Consider the shadow graph $\partial H$ of $H$. The number of edges in $\partial H$ is equal to $3 \abs{E(H)} = 3 \cdot \frac{nd}{3} = nd$. Then the average degree of a vertex in $\partial H$ is $d^{\partial H} = 2d$, and the maximum degree $\Delta^{\partial H}$ in $\partial H$ is at most $2 d_{max} \le 12d$ by Claim \ref{dmax}. Applying the Blakley-Roy inequality \cite{Blakley_Roy} to the graph $\partial H$, we obtain that there are at least $\frac{1}{2}n(d^{\partial H})^3$ (unordered) walks of length $3$ in $\partial H$. Then there are at least $$\frac{1}{2}n(d^{\partial H})^3-3n(\Delta^{\partial H})^2$$ paths of length $3$ in $\partial H$ as there are at most $3n(\Delta^{\partial H})^2$ walks that are not paths.
Indeed, if $v_1v_2v_3v_4$ is a walk that is not a path, then there exists a repeated vertex $v$ in the walk such that either $v_1 = v_3 =v$ or $v_2 = v_4=v$ or $v_1 = v_4=v$. Since $v$ can be chosen in $n$ ways and the other two vertices of the walk are adjacent to $v$, we can choose them in at most $(\Delta^{\partial H})^{2}$ different ways.


A path in $\partial H$ is called a rainbow path if the edges of the path are contained in distinct hyperedges of $H$. If a path $abcd$ is not rainbow then there are two (consecutive) edges in it that are contained in the same hyperedge of $H$. So there are two hyperedges $h, h' \in E(H)$, $h \cap h' \not = \emptyset$ such the path $abcd$ is contained in the $2$-shadow of $h, h'$. 
Now we estimate the number of non-rainbow paths.

We can choose these pairs $h, h' \in E(H)$ in $\sum_{v \in V(H)} \binom{d^H(v)}{2}$ ways and for a fixed pair $h,  h' \in E(H)$, it is easy to see that the path $abcd$  can chosen in $8$ different ways in the $2$-shadow of $h, h'$. 
%
%
%
Therefore, the number of non-rainbow paths in $\partial H$ is at most $$\sum_{v \in V(H)}  8\binom{d^H(v)}{2} \le 4n(d_{max})^2 \le 4n(6d)^2 = 144nd^2.$$ So the number of rainbow paths in $\partial H$ is at least $$\frac{1}{2}n(d^{\partial H})^3-3n(\Delta^{\partial H})^2 - 144nd^2 = \frac{1}{2}n(2d{\tiny })^3-3n(12d)^2-144nd^2 = 4nd^3 - 576nd^2.$$

Since each $3$-link in $H$ produces $4$ rainbow paths in $\partial H$, the number of rainbow paths in $\partial H$ is $4p(H)$. So,
$4p(H) \ge 4nd^3 - 576nd^2$. That is, $$p(H) \ge nd^3 - 144nd^2.$$ Combining this with \eqref{upper_bound}, we get 
$$
nd^3 - 144nd^2 \le p(H) \le (n+819d)\frac{nd}{3}.
$$
Simplifying, we get $d^2 - 144d \le (n+819d)/3$. That is, $$d \le \sqrt{\frac{n}{3}+\frac{173889}{4}} + \frac{417}{2}.$$ So, $$\abs{E(H)} = \frac{nd}{3} \le \frac{n}{3} \cdot \left (\sqrt{\frac{n}{3}+\frac{173889}{4}} + \frac{417}{2} \right) = \frac{1}{3 \sqrt3}n^{3/2} + O(n),$$ completing the proof of Theorem \ref{main}.

\section{$C_4$-free linear hypergraphs: Proof of Theorem \ref{c4}}
\label{c4section}

Let $H$ be a $3$-uniform linear hypergraph on $n$ vertices containing no (Berge) $C_4$. Let $d$ denote the average degree of a vertex in $H$.


\vspace{2mm}
\textbf{Outline of the proof:} Our plan is to first upper bound $\sum_{x \in N^H_1(v)} 2d(x)$ for each fixed $v \in V(H)$, which as the following claim shows, is not much more than $n$. Then we estimate $\sum_{v \in V(H)}\sum_{x \in N^H_1(v)} 2d(x)$ in two different ways to get the desired bound on $d$.

\begin{claim}
	\label{upperboundonsumofdegrees}
For every $v \in V(H)$, we have $$\sum_{x \in N^H_1(v)} 2d(x) \le n + 12d(v).$$
\end{claim}

\begin{proof}
First we show that most of the hyperedges incident to $x \in N^H_1(v)$ contain only one vertex from $N^H_1(v)$.

\begin{claim}
	\label{small_degreeG1}
For any given $x \in N^H_1(v)$,  the number of hyperedges $h \in E(H)$ containing $x$ such that $\abs{h \cap N^H_1(v)} \ge 2$ is at most $3$. 
\end{claim}

\begin{proof}

Suppose for a contradiction that there is a vertex $x \in N^H_1(v)$ which is contained in $4$ hyperedges $h$ such that $\abs{h \cap N^H_1(v)} \ge 2$. One of them is the hyperedge containing $x$ and $v$. Let $h_1, h_2, h_3$ be the other $3$ hyperedges. Then it is easy to see that two of these hyperedges intersect two different hyperedges incident to $v$, and these four hyperedges form a $C_4$ in $H$, a contradiction.  
\end{proof}

For each $x \in N^H_1(v)$, let $E_x = \{h \in E(H)  \mid h \cap N^H_1(v) = \{x\} \}$. Note that any hyperedge of $E_x$ does not contain $v$, so it contains exactly two vertices from $N^H_2(v)$. Let $S_x = \{w \in N^H_2(v) \mid \exists h \in E_x \text{ with } w \in h \}$. Then  $\abs{S_x} = 2 \abs{E_x}$ since $H$ is linear. Notice that $\abs{E_x} \ge d(x)-3$ by Claim \ref{small_degreeG1}, so

\begin{equation}
\label{boundonS_u}
\abs{S_x} \ge 2d(x)-6.
\end{equation}

The following claim shows that the sets $\{S_x \mid x \in N^H_1(v) \}$ do not overlap too much.

\begin{claim}
	\label{S_xsmeetless}
Let $x, y \in N^H_1(v)$ be distinct vertices. If $xyv \not \in E(H)$ then $S_x \cap S_y = \emptyset$ and if $xyv \in E(H)$ then $\abs{S_x \cap S_y} \le 2$.
\end{claim}

\begin{proof}
Take $x, y \in N^H_1(v)$ with $x \not = y$. Let $h_x, h_y \in E(H)$ be hyperedges incident to $v$ such that $x \in h_x$ and $y \in h_y$. 
First suppose $h_x \not = h_y$. Then it is easy to see that $S_x \cap S_y = \emptyset$ because otherwise $h_x, h_y$ and the two hyperedges containing $xw, yw$ for some $w \in S_x \cap S_y$ form a $C_4$, a contradiction.

Now suppose $h_x = h_y$. We claim that $\abs{S_x \cap S_y} \le 2$. Suppose for the sake of a contradiction that there are $3$ distinct vertices $v_1, v_2, v_3 \in S_x \cap S_y$. Then it is easy to see that there exist $i, j \in \{1,2,3\}$ such that neither $v_iv_jx$ nor $v_iv_jy$ is a hyperedge in $H$. So there are two different hyperedges $h_1, h_2 \in E_x$ such that $xv_i \in h_1$ and $xv_j \in h_2$. Similarly there are two different hyperedges $h'_1, h'_2 \in E_y$ such that $yv_i \in h'_1$ and $yv_j \in h'_2$. As $E_x \cap E_y = \emptyset$, the hyperedges $h_1, h_2, h'_1, h'_2$ are distinct and form a $C_4$, a contradiction.
\end{proof}

We will upper bound $\sum_{x \in N^H_1(v)} \abs{S_x}$. It follows from Claim \ref{S_xsmeetless} that each vertex $w \in N^H_2(v)$ belongs to at most two of the sets in $\{S_x \mid x \in N^H_1(v) \}$. Moreover, $w$ belongs to two sets $S_p, S_q \in \{S_x \mid x \in N^H_1(v) \}$ only if there exists a unique pair $p, q \in N^H_1(v)$ such that $pqv \in E(H)$ and for any such pair $p, q$ with $pqv \in E(H)$, there are at most $2$ vertices $w$ with $w \in S_p, S_q$. So there are at most $2d(v)$ vertices in $N^H_2(v)$ that are counted twice in the summation $\sum_{x \in N^H_1(v)} \abs{S_x}$. That is,

\begin{equation}
\label{sx}
\abs{N^H_2(v)} \ge \sum_{x \in N^H_1(v)} \abs{S_x} - 2d(v).
\end{equation}

As $N^H_2(v)$ and $N^H_1(v)$ are disjoint, we have $n \ge \abs{N^H_2(v)} + \abs{N^H_1(v)}$.  So by \eqref{sx}, 

\begin{equation}
n  \ge  \sum_{x \in N^H_1(v)} \abs{S_x} - 2d(v) +  \abs{N^H_1(v)} =    \sum_{x \in N^H_1(v)} \abs{S_x} - 2d(v)  + 2d(v) = \sum_{x \in N^H_1(v)} \abs{S_x}.
\end{equation}

Combining this with \eqref{boundonS_u}, we get

\begin{equation}
\label{lower_n}
n \ge \sum_{x \in N^H_1(v)} (2d(x)-6) = \sum_{x \in N^H_1(v)} 2d(x) - 6 \abs{N^H_1(v)} = \sum_{x \in N^H_1(v)} 2d(x) - 12d(v),
\end{equation}
completing the proof of Claim \ref{upperboundonsumofdegrees}. 
\end{proof}

%


We now estimate $\sum_{v \in V(H)} \sum_{x \in N^H_1(v)} 2d(x)$ in two different ways. On the one hand, by Claim \ref{upperboundonsumofdegrees}

\begin{equation}
\label{up}
\sum_{v \in V(H)} \sum_{x \in N^H_1(v)} 2d(x) \le \sum_{v \in V(H)} (n + 12d(v)) = n^2 + 12nd.
\end{equation}

On the other hand, 
\begin{equation}
\label{down}
\sum_{v \in V(H)} \sum_{x \in N^H_1(v)} 2d(x) = \sum_{v \in V(H)} 2d(v) \cdot 2d(v) = \sum_{v \in V(H)} 4d(v)^2 \ge 4nd^2.
\end{equation}

The last inequality follows from the Cauchy-Schwarz inequality. Finally, combining \eqref{up} and \eqref{down}, we get $4nd^2 \le n^2 + 12nd$. Dividing by $n$, we have $4d^2 \le n + 12d$, so $d \le \frac{1}{2} (\sqrt{n+9}+3)$. Therefore, $$\abs{E(H)} = \frac{nd}{3} \le \frac{1}{6} n\sqrt{n+9} + \frac{n}{2},$$ proving Theorem \ref{c4}.

\section{Proof of Theorem \ref{construction}: Construction}
\label{theorem4section}
We prove Theorem \ref{construction} by constructing a linear hypergraph $H$ below, and then we show that it is $\C^{\lin}_{2k+1}$-free. Finally, we count the number of hyperedges in it.

\vspace{3mm}

\textbf{Construction of $H$:} Let $G =(V(G), E(G))$ be a $\C_{2k-2}$-free bipartite graph (i.e., girth at least $2k$)  on $z$ vertices. Let the two color classes of $G$ be $L = \{l_1, l_2, \ldots l_{z_1}\}$ and $R = \{r_1, r_2, \ldots, r_{z_2}\}$ where $z = z_1+z_2$.

Now we construct a hypergraph $H = (V(H), E(H))$ based on $G$. Let $q$ be an integer. For each $ 1 \le t \le q$, let $L_t = \{l^t_1, l^t_2, \ldots, l^t_{z_1}\}$ and $R_t = \{r^t_1, r^t_2, \ldots, r^t_{z_2}\}$. Let $B = \{v_{i, j} \mid 1 \le i \le z_1, 1 \le j \le z_2\text{ and } l_ir_j \in E(G)\}$. (Note that $\abs{B} = \abs{E(G)}$ as we only create a vertex in $B$ if the corresponding edge exists in $G$.) Now let $V(H) = \bigcup\limits_{i=1}^{q} L_{i} \cup \bigcup\limits_{i=1}^{q} R_{i} \cup B$ and $E(H) = \{v_{i,j} l^t_i r^t_j \mid v_{i,j} \in B \text{ and } l_ir_j \in E(G) \text{ and } 1 \le t \le q\}$. Clearly $H$ is a linear hypergraph. 

\vspace{3mm}

\textbf{Proof that $H$ is $\C^{\lin}_{2k+1}$-free:}
Suppose for the sake of a contradiction that $H$ contains $C^{\lin}_{2k'+1}$, a linear cycle of length $2k'+1$ for some $k' \le k$. 

Since the basic cycle of $C^{\lin}_{2k'+1}$ is of odd length it must contain at least one vertex in $B$. (Note that here we used that the length of the linear cycle is odd.)

First let us assume that the basic cycle of $C^{\lin}_{2k'+1}$ contains exactly one vertex $x \in B$. Then $\bigcup\limits_{i=1}^{q} L_{i} \cup \bigcup\limits_{i=1}^{q} R_{i} \cup x$ contains all the basic vertices of $C^{\lin}_{2k'+1}$. For $X \subseteq V(H)$, let $H[X]$ denote the subhypergraph in $H$ induced by $X$. Notice that $x$ is a cut vertex in the $2$-shadow of $H[\bigcup\limits_{i=1}^{q} L_{i} \cup \bigcup\limits_{i=1}^{q} R_{i} \cup x]$. Therefore, there exists a $t$ such that the basic vertices of $C^{\lin}_{2k'+1}$ belong to $L_t \cup R_t \cup x$. Let $xu$ and $xv$ be the two edges incident to $x$ in the basic cycle of $C^{\lin}_{2k'+1}$. However, by construction the hyperedge containing $xu$ is the same as the hyperedge containing $xv$, which is impossible since $C^{\lin}_{2k'+1}$ is a linear cycle. Therefore, there are at least two basic vertices of $C^{\lin}_{2k'+1}$ in $B$. 

Let $c_1, c_2, \ldots, c_s$ be the basic vertices of $C^{\lin}_{2k'+1}$ in $B$ and let us suppose that they are ordered such that the subpaths $P_{i,i+1}$ of the basic cycle of $C^{\lin}_{2k'+1}$ from $c_i$ to $c_{i+1}$, are pairwise edge-disjoint for $1 \le i \le s$ (addition in the subscript is taken modulo $s$ from now on).  Note that $s \ge 2$ by the previous paragraph and $s \le k'$ because for each $i$, the subpath $P_{i,i+1}$ contains at least two edges. It is easy to see that for each $1 \le i \le s$, there exists a $t$ such that $V(P_{i,i+1}) \subseteq L_t \cup R_t \cup \{c_i, c_{i+1}\}$. 
Let $P'_{i,i+1}$ be a path in $G$ with the edge set $\{l_\alpha r_\beta \mid l^t_\alpha r^t_\beta \in E(P_{i,i+1}) \text{ for some } t\}$ for $1 \le i \le s$.  Clearly,  $\abs{E(P'_{i,i+1})} = \abs{E(P_{i,i+1})}-2 \ge 0$.
For each $c_i$, there exists $1 \le \alpha_i \le z_1$, $1 \le \beta_i \le z_2$ such that $c_i = v_{\alpha_i, \beta_i}$. Let $e_i = l_{\alpha_i} r_{\beta_i}$ for each $1 \le i \le s$, and let $e_i^t = l_{\alpha_i}^t r_{\beta_i}^t$ for each $ 1 \le t \le q$. 
Notice that $P'_{i,i+1}$ is a path in $G$  and $e_i \in E(G)$. Moreover, $P'_{i,i+1}$ is a path between a vertex of $e_i$ and a vertex of $e_{i+1}$ and if $E(P'_{i,i+1})= \emptyset$, then $e_i \cap e_{i+1} \not = \emptyset$. 

\begin{claim}
	\label{P'sareedgedisjoint}
The paths $P'_{i,i+1}$ (for $1 \le i \le s$) cannot contain any of the edges $e_j$ (for $1 \le j \le s$). Moreover, for any $1 \le i \not = j \le s$, the paths $P'_{i,i+1}$ and $P'_{j,j+1}$ are edge-disjoint. 
\end{claim}

\begin{proof}
Assume for the sake of contradiction a path $P'_{i,i+1}$ (for some $1 \le i \le s$) contains an edge $e_j$ (for some $1 \le j \le s$). This implies there exists $t$ with $1 \le t \le q$, such that $e_j^t$ is contained in $P_{i,i+1}$, so $e_j^t$ is contained the basic cycle of $C^{\lin}_{2k'+1}$. Then the (only) hyperedge containing $e_j^t$, namely $l_{\alpha_j}^t r_{\beta_j}^tv_{\alpha_j, \beta_j} = l_{\alpha_j}^t r_{\beta_j}^tc_j$ is a hyperedge of the linear cycle $C^{\lin}_{2k'+1}$. 
However, by definition of a linear cycle, the basic cycle must use exactly two vertices of any hyperedge of its linear cycle, a contradiction. Therefore the paths $P'_{i,i+1}$, $1 \le i \le s$, cannot contain any of the edges $e_j$ (for $1 \le j \le s$). 

Now we will show that for any $1 \le i \not = j \le s$, $P'_{i,i+1}$ and $P'_{j,j+1}$ are edge-disjoint. Suppose for a contradiction that $l_{\alpha}r_{\beta} \in E(P'_{i,i+1}) \cap E(P'_{j,j+1})$ for some $1 \le \alpha \le z_1$ and $1 \le  \beta \le z_2$. Then there exist $t \not = t'$ such that $l^t_{\alpha}r^t_{\beta}$ and $l^{t'}_{\alpha}r^{t'}_{\beta}$ are two disjoint edges of the basic cycle of $C^{\lin}_{2k'+1}$. However, $l^t_{\alpha}r^t_{\beta}v_{\alpha, \beta}, l^{t'}_{\alpha}r^{t'}_{\beta}v_{\alpha, \beta} \in E(H)$, which is impossible since the hyperedges containing disjoint edges of the basic cycle of a linear cycle must also be disjoint, by the definition of a linear cycle.
\end{proof}


Recall that by definition, the first vertex of $P_{j,j+1}$ is $c_j$. So the first edge of $P_{j,j+1}$ is contained in a hyperedge of the form $e_j^t \cup c_j$ for some $t$ (indeed all the hyperedges containing $c_j$ are of this form). This means the second vertex of $P_{j,j+1}$ is contained in $e_j^t$, so the first vertex of $P'_{j,j+1}$ is contained in $e_j$.  Similarly, the last vertex of $P'_{j-1,j}$ is also contained in $e_j$. Therefore, the last vertex of $P'_{j-1,j}$ and the first vertex of $P'_{j,j+1}$ are both contained in $e_j$. If these vertices are different, then we call $e_j$ a \emph{connecting edge}. So using Claim \ref{P'sareedgedisjoint}, the edges of $\cup_i E(P'_{i,i+1})$ together with the connecting edges form a circuit $\mathcal C$ in $G$ (i.e., a cycle where vertices may repeat but edges do not repeat). 

Now we claim that $\mathcal C$ is non-empty and contains at most $2k-1$ edges. Indeed, the number of edges of $\mathcal C$ is at least $\sum_{i=1}^s \abs{E(P'_{i,i+1})}$. Moreover, as the number of connecting edges is at most $s$, the number of edges in $\mathcal C$ is at most $\sum_{i=1}^s\abs{E(P'_{i,i+1})} + s$. 
Since $\sum_{i=1}^s\abs{E(P'_{i,i+1})} = \sum_{i=1}^s \abs{E(P_{i,i+1})} -2s = 2k'+1-2s$, and $2 \le s \le k'$, it is easily seen that $\mathcal C$ is non-empty and contains at most $2k'+1-s \le 2k'-1 \le 2k-1$ edges, as claimed. (Let us remark that here the fact that the length of the linear cycle $C^{\lin}_{2k'+1}$ is odd played a crucial role in ensuring that the circuit $\mathcal C$ is non-empty --indeed, if the length is even, it is possible that $E(P'_{i,i+1})$ is empty for each $i$.)

Since every non-empty circuit contains a cycle, we obtain a cycle of length at most $2k-1$ in $G$, a contradiction, as desired.

\vspace{3mm}

\textbf{Bounding $\ex^{\lin}_3 (n, \C^{\lin}_{2k+1})$ from below: } 
We assumed $\ex_{\bip}(z, \C_{2k-2}) \ge  (1 + o(1)) c (z/2)^{\alpha}$ for some $c, \alpha > 0$. So there is a $\C_{2k-2}$-free bipartite graph $G$ on $z$ vertices with \begin{equation}
\label{sizeofG}
\abs{E(G)} =  (1 + o(1)) c \left(\frac{z}{2}\right)^{\alpha}.
\end{equation}

Let $H$ be the $\C^{\lin}_{2k+1}$-free hypergraph constructed based on $G$ (as described in the Construction above). Then the number of hyperedges in $H$ is $\abs{E(G)} \cdot q$. So we have


\begin{equation}
\label{const_lower}
\ex^{\lin}_3 (n, \C^{\lin}_{2k+1}) \ge \abs{E(H)} = \abs{E(G)} \cdot q \ge \abs{E(G)} \cdot  \left \lfloor\frac{n - \abs{E(G)}}{z} \right \rfloor.
\end{equation}

%

Substituting \eqref{sizeofG} in \eqref{const_lower} and choosing  $z = (1 + o(1)) \left(\frac{2^{\alpha}(\alpha-1)}{c(2\alpha -1)}\right)^{\frac{1}{\alpha}} n^{\frac{1}{\alpha}}$, we obtain that 

$$\ex^{\lin}_3 (n, \C^{\lin}_{2k+1})  \ge (1 + o(1)) \frac{\alpha c}{4 \alpha - 2} \cdot \left(\frac{\alpha -1}{c(2 \alpha -1)}\right)^{1 - \frac{1}{\alpha}} n^{2 - \frac{1}{\alpha}},$$
completing the proof of Theorem \ref{construction}.

\section*{Acknowledgment}

We are grateful to the two anonymous referees for their detailed and very helpful comments.

The research of the authors is partially supported by the National Research, Development and Innovation Office – NKFIH, grant K116769.


\begin{thebibliography}{99}

\bibitem{AKSV}
P. Allen, P. Keevash, B. Sudakov and J. Verstra\"ete. ``Tur\'an numbers of bipartite graphs plus an odd cycle." \textit{Journal of Combinatorial Theory, Series B} 106 (2014): 134--162.


\bibitem{Benson}
C. T. Benson. ``Minimal regular graphs of girths eight and twelve." \textit{Canad. J.
Math.} 18 (1966), 1091--1094.


\bibitem{BGY2008} 
B. Bollob\'as and E. Gy\H{o}ri. ``Pentagons vs. triangles." \textit{Discrete Mathematics}, 308 (19) (2008), 4332--4336.


\bibitem{Brown_Erd_Sos}
W.G. Brown, P. Erd\H os and V. S\'os. ``On the existence of triangulated spheres in 3-graphs and related problems." \textit{Periodica Mathematica Hungaria} 3 (1973), 221--228.


\bibitem{Bondy_Simonovits}
A. Bondy and M. Simonovits. ``Cycles of even length in graphs." \textit{Journal of Combinatorial Theory, Series B} 16.2 (1974): 97--105.

\bibitem{Blakley_Roy}
R. G. Blakley and P. Roy. ``A H\"{o}lder type inequality for symmetric matrices with nonnegative entries." \textit{Proceedings of the American Mathematical Society} 16.6 (1965): 1244--1245.


\bibitem{bukhJiang}
B. Bukh and Z. Jiang. ``A bound on the number of edges in graphs without an even cycle." \textit{Combinatorics, Probability and Computing} (2014): 1--15.

\bibitem{C_Graber_Jiang}
C. Collier-Cartaino, N. Graber and T. Jiang. ``Linear Tur\'an numbers of $r$-uniform linear cycles and related Ramsey numbers." \textit{Combinatorics, Probability and Computing} 27.3 (2018): 358--386.

\bibitem{Erdos}
P. Erd\H{o}s. ``Some recent progress on extremal problems in graph theory." \textit{Congr. Numer.} 14 (1975), 3--14.

\bibitem{Erd_Sim}
P. Erd\H{o}s and M. Simonovits. ``Compactness results in extremal graph theory." \textit{Combinatorica} 2 (1982), no. 3, 275--288.

\bibitem{C_5_C_4}
B. Ergemlidze, E. Gy\H{o}ri, A. Methuku, C. Tompkins and N. Salia. ``On $3$-uniform hypergraphs avoiding a cycle of length four." (In preparation.)

\bibitem{Furedi_Ozkahya}
Z. F\"uredi and L. \"Ozkahya. ``On 3-uniform hypergraphs without a cycle of a given length." \textit{Discrete Applied Mathematics}, 216 (2017): 582--588.

\bibitem{gp1}   D. Gerbner and C. Palmer. ``Extremal results for Berge-hypergraphs." \textit{SIAM Journal on Discrete Mathematics}, 31.4 (2017): 2314--2327.

\bibitem{GMM} D. Gerbner, A. Methuku and M. Vizer. ``Asymptotics for the Tur\'an number of Berge-$K_{2,t}$." arXiv preprint arXiv:1705.04134 (2017).

\bibitem{Gyori_Lemons}
E. Gy\H{o}ri and N. Lemons. ``3-uniform hypergraphs avoiding a given odd cycle." \textit{Combinatorica} 32.2 (2012): 187--203.



\bibitem{Lazebnik_U_W}
F. Lazebnik, V. A. Ustimenko and A. J. Woldar. ``A new series of dense graphs of high girth." \textit{Bull. Amer. Math. Soc.} 32 (1995), no. 1, 73--79.

\bibitem{Lazeb_Verstraete}
F. Lazebnik and J. Verstra\"ete. ``On hypergraphs of girth five." Electron. J. Combin 10 (2003): R25.

\bibitem{pihurko}
O. Pikhurko. ``A note on the Tur\'an function of even cycles." \textit{Proceedings of the American Mathematical Society} 140.11 (2012): 3687--3692.

\bibitem{Ruzsa_Szem}
I. Ruzsa and E. Szemer\'edi. ``Triple systems with no six points carrying three triangles." in \textit{Combinatorics, Keszthely, Colloq. Math. Soc. J. Bolyai} 18, Vol II (1976): 939--945.

\bibitem{Singleton}
R. R. Singleton. ``On minimal graphs of maximum even girth." \textit{J. Combinatorial Theory} 1 (1966), 306--332.

\bibitem{T2016} C. Timmons. ``On $r$-uniform linear hypergraphs with no Berge-$K_{2,t}$." arXiv preprint arXiv:1609.03401 (2016).

\bibitem{verstraete}
J. Verstra\"ete. ``On arithmetic progressions of cycle lengths in graphs." \textit{Combinatorics, Probability and Computing} 9.04 (2000): 369--373.


\end{thebibliography}
\end{document}